\documentclass[a4paper, 11pt, reqno]{amsart}
\usepackage[english]{babel}
\usepackage[utf8]{inputenc}
\usepackage{amsmath}
\usepackage{amssymb}
\usepackage{mathtools} %Needed for \coloneqq
\usepackage{color}
\usepackage{nicefrac}
\DeclareMathOperator{\sgn}{sgn}
\usepackage{tikz-cd}
\usepackage{listings}
\usepackage{float}
\usepackage{a4wide}
\usepackage{amsrefs}
\usepackage{amsthm}
\usepackage{mathrsfs}
\usepackage{comment}
\usepackage{leftindex}
\usepackage{mathtools}
\usepackage{hyperref}

\theoremstyle{plain}
\newtheorem{thm}{Theorem}[section]
\newtheorem{prop}[thm]{Proposition}
\newtheorem{lemma}[thm]{Lemma}

\theoremstyle{definition}
\newtheorem{defn}[thm]{Definition}

\theoremstyle{remark}
\newtheorem*{rmk}{Remark}

\newcommand{\C}{\mathbb{C}}
\renewcommand{\H}{\mathbb{H}}
\newcommand{\Z}{\mathbb{Z}}

\newcommand{\N}{\mathbb{N}}

\newenvironment{proofa}{%
  \proof}{\endproof}

\makeatletter
\let\@@pmod\pmod
\DeclareRobustCommand{\pmod}{\@ifstar\@pmods\@@pmod}
\def\@pmods#1{\mkern4mu({\operator@font mod}\mkern 6mu#1)}
\makeatother

\title{Local Weak Maass Forms and Traces of Cycle Integrals}
\author{Kilian Rausch}
\author{Johann Stumpenhusen}
\address{Department of Mathematics and Computer Science, Division of Mathematics, University of Cologne, Weyertal 86-90, 50931 Cologne, Germany}
\email{krausch1@uni-koeln.de}
\email{jstumpen@math.uni-koeln.de}
\date{\today}
\begin{document}
 
\begin{abstract}
    Using relations established by Bringmann and Kane in a recent preprint based on the function $\omega_{k+1,D}$ studied by Mono, Rolen, and the second author, we introduce two further functions naturally associated with a construction by Bringmann and Mono. We exploit results by Löbrich--Schwagenscheidt, Mono, and the second author on representations of these functions as traces of cycles integrals in order to provide more identities of the same kind.
\end{abstract} 

\maketitle

\section{Introduction and Statement of Results}

\subsection{Modularly transforming forms related to real quadratic fields} Recently, Bringmann and Kane \cite{BringKane} simplified and extended results by Mono, Rolen, and the second author \cite{MonoRolenStum} concerning the functions
\[
\omega_{k+1,D}(z) \coloneqq \sum_{Q \in \mathcal{Q}_D} \frac{Q_z}{Q(z,1)^{k+1}}, \qquad k > 2, \, z = x + iy \in \H,
\]
where $D > 0$ is a fundamental discriminant, $\mathcal{Q}_D \coloneqq \{aX^2 + bXY + cY^2 : D = b^2 - 4ac\}$ is the set of all binary quadratic forms of discriminant $D$, and $Q_z \coloneqq y^{-1}(a|z|^2 + bx + c)$. These functions were inspired by the weight $2k$ holomorphic cusp forms
\[
f_{k,D}(z) \coloneqq \sum_{Q \in \mathcal{Q}_D} \frac{1}{Q(z,1)^k}, \qquad k > 2,\, z \in \H,
\]
which were introduced by Zagier \cite{zagier75}*{Appendix A} in 1975. The main result in \cite{MonoRolenStum} states that $\omega_{k+1,D}$ is a weight $2k + 2$ weak Maass form and there exists a splitting of $\omega_{k +1 ,D}$ into a holomorphic and a non-holomorphic part via
\[
\omega_{k+1,D}(z) = - i\sum_{Q \in \mathcal{Q}_D}\frac{Q'(z,1)}{Q(z,1)^{k+1}} + \frac{1}{y}f_{k,D}(z), \qquad Q'(z,1) \coloneqq \frac{\partial}{\partial z} Q(z,1).
\]

\subsection{Maass operators} In the theory of weak Maass forms, two differential operators play a crucial role: let $\kappa \in \Z$ and define the \emph{Maass raising operator of weight $\kappa,$} $R_\kappa \coloneqq  2i\frac{\partial}{\partial z} + \frac{\kappa}{y}$ as well as as the \emph{Maass lowering operator of weight $\kappa,$} $L_\kappa \coloneqq -2iy^2 \frac{\partial}{\partial \overline{z}}$. Bringmann and Kane showed that the connection between $f_{k,D}$ and $\omega_{k+1,D}$ is even tighter.
\begin{thm}[\cite{BringKane}*{Theorem 1.2}]\label{thm:BringKaneOmegaFRelation}
    It holds that
    \[\omega_{k+1,D}(z) = \frac{R_{2k}f_{k,D}(z)}{2k}\]
    for $z \in \H$ and $k > 2$ even.
\end{thm}

Let $\xi_\kappa \coloneqq -y^{\kappa - 2}\overline{L_\kappa}$ be the weight $\kappa$ \emph{shadow operator} and $\mathbb{D} \coloneqq \frac{\partial}{2\pi i \partial z}$ be another differential operator. For $\kappa \in \N$, we call $\mathbb{D}^{\kappa-1}$ the \emph{Bol operator} of weight $2 - \kappa$ which satisfies Bol's identity
\begin{equation*}\label{eq:BolsIdentity}
    \mathbb{D}^{\kappa-1} = (-4\pi)^{1-\kappa}R_{2-\kappa}^{\kappa-1}.
\end{equation*}

A decade ago, Bringmann, Kane, and Kohnen introduced the locally harmonic Maass form (see Subsection \ref{subsec:LocalMaass} for a definition)
\[\mathcal{F}_{1-k,D}: \H \to \C, \quad z \mapsto \frac{1}{2}\sum_{Q \in \mathcal{Q}_D}\sgn(Q_z)Q(z,1)^{k-1}\beta\left(\frac{Dy^2}{|Q(z,1)|^2};k - \frac{1}{2},\frac{1}{2}\right)\]
where
\[\beta(x;r,s) := \int_0^x t^{r-1}(1-t)^{s-1}\mathrm{d}t\]
is the incomplete $\beta$-function. It is a local preimage of $f_{k,D}$ under both the shadow and the Bol operator. Define
\[
\mathcal{W}_{-k,D}(z) := y^{2k}\overline{f_{k,D}(z)}.
\]
Bringmann and Kane also proved the following.
\begin{thm}[\cite{BringKane}*{Theorem 1.3}] \label{thm:RelationWf}
    Let $k > 2$ be even.
    \begin{enumerate}
        \item We have
        \[\mathcal{W}_{-k,D}(z) = D^{\frac{1}{2}-k}L_{2-2k}\mathcal{F}_{1-k,D}(z)\]
        for $z \in \H \setminus E_D$.
        \item Furthermore, it holds 
        \[\xi_{-2k}\mathcal{W}_{-k,D}(z) = 2k\omega_{k+1,D}(z)\]
        for $z \in \H$.
    \end{enumerate}
\end{thm}

\begin{rmk}
    Being a locally harmonic Maass form, $\mathcal{F}_{1-k,D}$ satisfies $\Delta_{2-2k}\mathcal{F}_{1-k,D}(z) = 0$ where $\Delta_\kappa$ is the usual weight $\kappa$ Laplacian. Since $\Delta_\kappa = - R_{\kappa - 2} \circ L_\kappa$, we derive that $R_{-2k}\mathcal{W}_{-k,D}(z) = 0$ and furthermore, due to Bol's identity, $\mathbb{D}^{2k+1}\mathcal{W}_{-k,D}(z) = 0$. In fact, this follows from the decomposition of $\mathcal{F}_{1-k,D}$ into a linear combination of the holomorphic and anti-holomorphic Eichler integrals of $f_{k,D}$ as well as a period polynomial as shown in \cite{bkk}*{Theorem 1.3}.
\end{rmk}

\subsection{Cycle integrals}  Let $f$ be a smooth function on $\H$ transforming like a weight $\kappa$ modular form. The weight $\kappa$ cycle integral of $f$ with a binary quadratic form $Q$ is defined as
\[\mathcal{C}_\kappa(f(\tau),Q) := D^{\frac{1}{2}-\frac{\kappa}{4}}\int_{\Gamma_Q\backslash S_Q}f(\tau)Q(\tau,1)^{\frac{\kappa}{2}-1}\mathrm{d}\tau\]
where $S_Q$ is the semi-circle connecting the roots of the quadratic polynomial $Q(z,1)$ called \emph{Heegner geodesic} and the orientation of the integral is clockwise if
\[0 < \sgn([a,b,c]) := \begin{cases}
    \sgn(a) &\text{if }a \neq 0,\\
    \sgn(c) &\text{if }a = 0,
\end{cases}\]
and otherwise it is counter-clockwise.

A central aspect of cycle integrals is their appearance as hyperbolic analogues of Fourier coefficients which is encoded in the following theorem by Kohnen and Zagier.

\begin{lemma}[\protect{\cite{koza84}*{Section 3}, \cite{koh85}*{Proposition 7}}] \label{lem:HyperbolicPeriods}
Let $f \in S_{2k}(\Gamma)$. Then, we have
\begin{align*}
    \langle f, f_{k,D} \rangle = \pi \binom{2k-2}{k-1}2^{2-2k}D^{-\frac{k}{2}} \sum_{Q \in \mathcal{Q}_D \slash \Gamma} \mathcal{C}_{2k}(f,Q).
\end{align*}
\end{lemma}

As proven by Löbrich and Schwagenscheidt \cite{lsmeromorphic}*{Theorem 4.2}, the functions $\mathcal{F}_{1-k,D}$ may be represented as traces of cycle integrals of a certain family of bivariate functions that transform like modular forms in both variables. Here, we set $\tau = u + iv$.

\begin{defn}\label{def:GeneralPeterssonPoincare}
    For $z, \tau \in \H$, $k_1 \in \N_{\geq 2}$ and $k_2 \in \Z$, we set
    \begin{align*}
        H_{k_1,k_2}(z,\tau) &:= \sum_{\gamma \in \Gamma} \frac{v^{k_1+k_2}}{(z - \tau)^{k_1-k_2}(z - \overline{\tau})^{k_1+k_2}}\Big|_{2k_1,z}\gamma\\
        &= \sum_{\gamma \in \Gamma} \frac{v^{k_1+k_2}}{(z - \tau)^{k_1-k_2}(z - \overline{\tau})^{k_1+k_2}}\Big|_{-2k_2,\tau}\gamma.
    \end{align*}
     We also define
    \begin{align*}
        \mathscr{P}_{k_1,k_2}(z,\tau) &\coloneqq \sum_{\gamma \in \Gamma} \frac{v^{k_1+k_2}}{\left(z+\tau\right)^{k_1-k_2}\left(z+\overline{\tau}\right)^{k_1+k_2}}\Big\vert_{2k_1,z}\gamma\\
        &= \sum_{\gamma \in \Gamma} \frac{v^{k_1+k_2}}{\left(z+\tau\right)^{k_1-k_2}\left(z+\overline{\tau}\right)^{k_1+k_2}}\Big\vert_{-2k_2,\tau}\gamma.
    \end{align*}
\end{defn}

Löbrich and Schwagenscheidt showed that
\begin{equation}\label{eq:CycleIntegralsCurlyF}
    \mathcal{F}_{1-k,D}(z) = \frac{D^{\frac{k}{2}-\frac{1}{2}}}{2}\binom{2k - 2}{k - 1}\sum_{Q \in \mathcal{Q}_D \slash \Gamma}\mathcal{C}_{2k}(H_{k,k-1}(\tau,z),Q)
\end{equation}
while Mono \cite{mo21}*{item (3) of Theorem 1.1} obtained
\begin{equation*}\label{eq:CycleIntegralsSmallG}
g_{k+1,D}(z) = \frac{2\Gamma(2k + 2)\prod_{j = 0}^{k-1}\left(k + j + 1)\right)}{D^{\frac{k+1}{2}}\Gamma\left(\frac{k+1}{2}\right)^2\prod_{l = 1 - k}^{-1}\left[(1 + l)(-l) + k(k + 1)\right]}
\sum_{Q \in \mathcal{Q}_D \slash \Gamma} \mathcal{C}_{-2k}(H_{k+1,k}(z,\tau),Q)
\end{equation*}
where
\[
g_{k+1,D}(z) \coloneqq \sum_{Q \in \mathcal{Q}_D} \frac{\sgn(Q_z)}{Q(z,1)^{k+1}}, \qquad k > 2, z \in \H,
\]
is another twisted variant of $f_{k,D}$ that is locally holomorphic and modular of weight $2k + 2$. Note the remarkable similarity of both representations as mentioned in Mono's dissertation \cite{MonoThesis}.

Recently, the second author \cite{Stum26}*{Theorem 1.2} established that also
\begin{equation*}\label{eq:CycleIntegralsSmallF}
    f_{k,D}(z) = \frac{(2k - 1)}{D^{\frac{k}{2}}} \binom{2k - 2}{k - 1} \sum_{Q \in \mathcal{Q}_D/\Gamma}\mathcal{C}_{2k}\left(\mathscr{P}_{k,-k}(z, \tau), Q\right)
\end{equation*}
which will be useful later on.

\subsection{Results} A few years ago, Bringmann and Mono \cite{brimo} constructed the function $\mathcal{G}_{-k,D}$ (see Section \ref{sec:ExtendingTheDiagram} for a definition) as a preimage of $g_{k+1,D}$ under both the shadow and the Bol operator, imitating the construction of $\mathcal{F}_{1-k,D}$. Motivated by Bringmann and Kane's recent extensions, we introduce another two functions.

\begin{defn}\label{defn:CurlyJIota}
    Let $k > 2$ and $D > 0$ a discriminant. We define
    \begin{equation*}\label{defeq:Iota}
        \iota_{k+2,D}(z) \coloneqq \frac{R_{2k+2}g_{k+1,D}(z)}{2k + 2}
    \end{equation*}
    as well as
    \begin{equation*}\label{defeq:CurlyJ}
        \mathcal{J}_{-1-k,D}(z) \coloneqq y^{2k+2}\overline{g_{k+1,D}(z)}.
    \end{equation*}
\end{defn}

Our first result is further extending the diagrams from \cite{BringKane}, \cite{MonoRolenStum}, and the end of Chapter I in Mono's PhD thesis \cite{MonoThesis}*{Subchapter I.3}.

\begin{figure}[H]
    \begin{center}% JS: Die neuen Resultate gerne als unterbrochene Linien wegen Inklusion
    \begin{tikzcd}[column sep=0.55in,row sep=0.55in]
        \mathcal{W}_{-k,D} \arrow[dd, swap, bend right=15, "\xi_{-2k}"] & & \mathcal{F}_{1-k,D} \arrow[rrdd, leftrightarrow, "\substack{\mathbb{P}_{2k+2} \textup{ and} \\ \mathbb{P}_{2k} \textup{ resp.}}" description] \arrow[ll, swap, "L_{2-2k}"] \arrow[dd, bend left=15, "\mathbb{D}^{2k-1}"] \arrow[dd, bend right = 15, swap, "\xi_{2-2k}"] &  & \mathcal{G}_{-k,D} \arrow[ll, swap, leftrightarrow, "\textup{twist by } \sgn(Q_z)"] \arrow[rr, dashed, "L_{-2k}"] \arrow[dd, bend left=15, "\mathbb{D}^{2k+1}"] \arrow[dd, bend right = 15, swap, "\xi_{-2k}"] & & \mathcal{J}_{-1 -k,D} \arrow[dd, dashed, bend right=15, swap, "\xi_{-2-2k}"]\\
        &&&&&&\\
        \omega_{k + 1,D} \arrow[uu, bend right=15, swap, "\xi_{2k+2}"] & & f_{k,D} \arrow[ll, swap, "R_{2k}"] \arrow[rr, swap, leftrightarrow, "\textup{twist by } \sgn(Q_z)"] \arrow[lluu, leftrightarrow, "\textup{weight flip}" description] \arrow[rd, swap, "L_{2k}"]& & g_{k+1,D} \arrow[rruu, leftrightarrow, dashed, "\textup{weight flip}" description] \arrow[rr, dashed, "R_{2k+2}"] \arrow[ld, "L_{2k+2}"] & & \iota_{k+2,D} \arrow[uu, bend right=15, dashed, swap, "\xi_{2k+4}"] \\
        &&& 0 &&&
    \end{tikzcd}
    \caption{Relations between the functions considered in this paper. Dashed arrows indicate results established in this paper.}
    \end{center}
\end{figure}
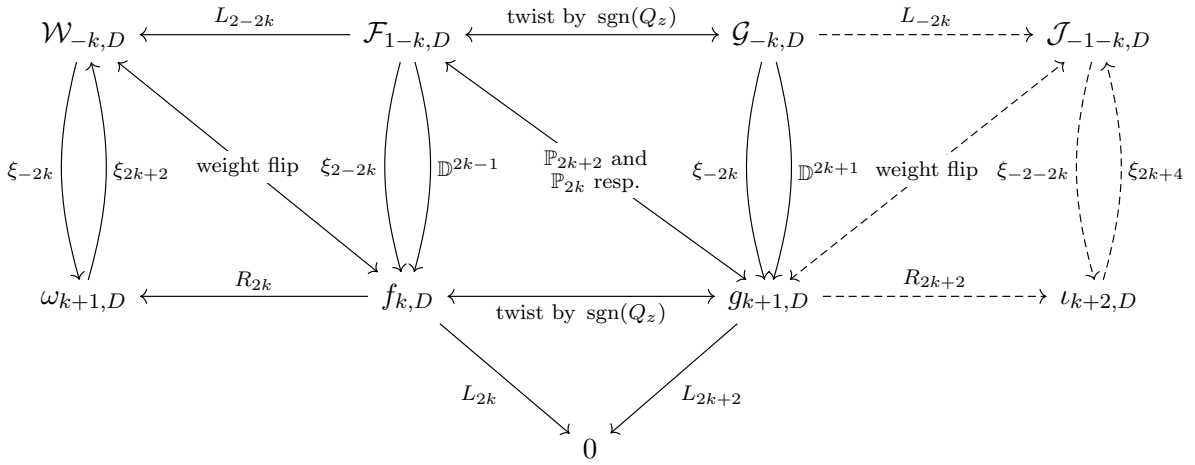

Note the indicated direct link between $\mathcal{F}_{1-k,D}$ and $g_{k+1,D}$. We will show that, using the relations given in the diagram above, generalising to $H_{k_1,k_2}$ and $\mathscr{P}_{k_1,k_2}$ instead of $\mathbb{P}_{2k}$ as an input for cycle integrals allows us to extend the representation as traces of cycle integrals to almost all functions in this family.

\begin{thm}\label{thm:MainResultWeakMaassForms}
    Let $k > 2$ be even and $D > 0$ a discriminant. We have
	\begin{align*}
	\mathcal{W}_{-k,D}(z) &= -\frac{ky^{2k}}{2D^{\frac{k}{2}}} \binom{2k}{k}
	\sum_{Q \in \mathcal{Q}_D/\Gamma} \mathcal{C}_{-2k}\left(v^{2k}\overline{\mathscr{P}_{k,-k}\left( z, \tau \right)},Q\right)
	\intertext{as well as}
	\omega_{k+1,D}(z) &= \frac{k}{2D^{\frac{k}{2}}} \binom{2k}{k} \sum_{Q \in \mathcal{Q}_D/\Gamma}\mathcal{C}_{2k}\left(\mathscr{P}_{k,-(k+1)}(\tau,z), Q\right)
	\intertext{for $z \in \H$. For $z \in \H \setminus E_D$, we also have}
	\mathcal{W}_{-k,D}(z) &= \frac{k}{4D^{\frac{k}{2}}}\binom{2k}{k}\sum_{Q \in \mathcal{Q}_D/\Gamma}\mathcal{C}_{2k}\left(H_{k,k}(\tau,z),Q\right).
	\end{align*}
\end{thm}

In a similar manner, we obtain a representation for the newly introduced $\mathcal{J}_{-1-k,D}$.

\begin{thm}\label{thm:MainResultLocallyWeakMaassForms}
    Let $k > 2$ be even and $D > 0$ a discriminant. We have
	\begin{align*}
	\mathcal{J}_{-1-k,D}(z) &=\frac{-y^{2k+2}2k\Gamma(2k+2)\prod_{j=0}^{k-1}(k+j+1)}{D^{\frac{k+1}{2}}\Gamma\left(\frac{k+1}{2}\right)^2 \prod_{l=1-k}^{-1}((1+l)(-l)+k(k+1))}\\
	&\phantom{====================} 
	\times \sum_{Q \in \mathcal{Q}_D / \Gamma} \mathcal{C}_{2k}\left(v^{-2k}\overline{H_{k+1,k}(z, \tau)},Q\right)
	\end{align*}
	for $z \in \H \setminus E_D$.
\end{thm}

Lastly, we recover another representation for $f_{k,D}$ and $\omega_{k+1,D}$ via Bol's identity.

\begin{thm}\label{thm:NewRepForFkD}
	Let $k > 2$ be even and $D > 0$ a discriminant. We have
	\begin{align*}
	f_{k,D}(z) &= \frac{k}{4D^{\frac{k}{2}}}\binom{2k}{k}\sum_{Q \in \mathcal{Q}_D \slash \Gamma}\mathcal{C}_{2k}(H_{k,-k}(\tau,z),Q)
	\intertext{and}
	\omega_{k+1,D}(z) &= \frac{k}{4D^{\frac{k}{2}}}\binom{2k}{k}\sum_{Q \in \mathcal{Q}_D \slash \Gamma}\mathcal{C}_{2k}(H_{k,-(k+1)}(\tau,z),Q)
	\end{align*}
	for $z \in \H \setminus E_D$.
\end{thm}

This paper is structured in the following way. In Section 2 we give some preliminaries. Section 3 deals with the extension of the diagram in figure 1. In the fourth and final section, we derive the cycle integral representations.
\section*{Acknowledgements}
The authors thank Kathrin Bringmann and Andreas Mono for insightful comments on earlier versions of this manuscript. 
The first author is funded by the European
Research Council (ERC) under the European Union's Horizon 2020 research and innovation programme (grant agreement No. 101001179).
\section{Preliminaries}

If not clarified by the context, we denote the variable with respect to which an operator is considered to be by an index.

\subsection{Weak Maass forms}\label{subsec:LocalMaass} Let $\kappa \in \Z$, $f$ a function defined on a dense subset of $\H$, and $\gamma = \left(\begin{smallmatrix}
    a_1 & a_2 \\ a_3 & a_4
\end{smallmatrix} \right)\in \Gamma$, then the \emph{weight $\kappa$ slash operator} of $\gamma$ on $f$ is defined via
\[(f\big\vert_\kappa\gamma)(z) := j(\gamma,z)^{-\kappa}f\left(\frac{a_1z + a_2}{a_3z + a_4}\right),\]
where $j(\gamma,z) := (a_3z + a_4)$ is the usual modular cocycle.
The raising and lowering operator both commute with the weight $\kappa$ slash operator in the following sense:
\begin{align*}
    R_\kappa f \big|_{\kappa+2} \gamma = R_{\kappa } \left(f\big|_{\kappa}\gamma\right) 
    \quad \text{ and }\quad
        L_\kappa f \big|_{\kappa-2} \gamma = L_{\kappa} \left(f\big|_{\kappa}\gamma\right).
\end{align*}
Let $N \in \N$. The function $f$ is called a \emph{weight $\kappa$ weak Maass form for}
\[\Gamma_0(N) := \left\{\begin{pmatrix}
        a_1 & a_2 \\ a_3 & a_4
    \end{pmatrix}
\in \Gamma : a_4 \equiv 0 \pmod*{N} \right\}\]
\emph{of manageable growth with eigenvalue $\lambda$ } if
\begin{enumerate}
    \item for $z \in \H$, we have
    \[(f\big\vert_\kappa\gamma)(z) = f(z)\]
    for all $\gamma \in \Gamma_0(N)$,
    \item with $\Delta_\kappa \coloneqq -y^2 \left(\frac{\partial^2}{\partial x^2} + \frac{\partial^2}{\partial y^2}\right) + i\kappa y \left(\frac{\partial}{\partial x} + i\frac{\partial}{\partial y}\right)$ being the weight $\kappa$ hyperbolic Laplacian operator, we have $\Delta_\kappa f = \lambda f$ on $\H$,
    \item it holds that 
    \[f(z) = O(e^{\varepsilon y})\]
    for some $\varepsilon > 0$ as $y \to \infty$, and analogous condition hold at every other cusp of $\Gamma_0(N)$.
\end{enumerate}
If $\lambda = 0$, $f$ is called \emph{harmonic}, and the vector space of weight $\kappa$ harmonic Maass forms  of manageable growth is denoted by $H_\kappa^{!}(N)$.

The function $f$ is called a \emph{weight $\kappa$ local weak Maass form of manageable growth for $\Gamma_0(N)$ with eigenvalue $\lambda$} if there exists an exceptional set $E \subset \H$ of measure 0 such that
\begin{enumerate}
    \item for $z \in \H \setminus E$, we have
    \[(f\big\vert_\kappa\gamma)(z) = f(z)\]
    for all $\gamma \in \Gamma_0(N)$,
    \item we have $\Delta_\kappa f = \lambda f$ on $\H \setminus E$,
    \item for $z \in E$, we have
    \[\lim_{\varepsilon \searrow 0} \frac{1}{2}\left(f(z + i\varepsilon) + f(z - i\varepsilon)\right) = f(z),\]
    \item the function $f$ has at most exponential growth towards $i\infty$.
\end{enumerate}
If a (local) weak Maass form vanishes towards $i \infty$, it is called \emph{cuspidal} or a \emph{cusp form}.

An overview can be found in \cite{bfkr} for example.

\subsection{Differential operators and Eichler integrals}\label{subsec:DifferentialOperators} Let $\kappa, N \in \N$. The Bol operator
\[\mathbb{D}^{\kappa-1}: H_{2-\kappa}^!(\Gamma_0(N)) \to M_\kappa^!(\Gamma_0(N)), \qquad f \mapsto \mathbb{D}^{\kappa-1}(f)\]
preserves modularity. Likewise, the shadow operator
\[\xi_{2-\kappa}: H_{2-\kappa}^!(\Gamma_0(N)) \twoheadrightarrow M_\kappa^!(\Gamma_0(N)), \qquad f \mapsto 2iy^{2-\kappa}\overline{\frac{\partial}{\partial \overline{z}}f}\]
also preserves modularity and we call $\xi_{2-\kappa}(f)$ the \emph{shadow} of $f$. Each of these two operators is accompanied by the so-called \emph{holomorphic} or \emph{non-holomorphic Eichler integral}, respectively, which are defined, for $f \in S_\kappa(\Gamma_0(N)),$ as
\begin{align*}
    \mathcal{E}_f(z) &\coloneqq -\frac{(2\pi i)^{\kappa-1}}{(\kappa-2)!}\int_z^{i\infty}f(\omega)(z - \omega)^{\kappa-2} \, \textup{d}\omega \quad \textup{and}\\
    f^*(z) &\coloneqq (2i)^{1-\kappa}\int_{-\overline{z}}^{i\infty}\overline{f(-\overline{\omega})}(\omega + z)^{\kappa-2} \, \textup{d}\omega.
\end{align*}
They satisfy the relations
\[\mathbb{D}^{\kappa-1}(\mathcal{E}_f) = f \quad \textup{and} \quad \xi_{2-\kappa}(f^*) = f\]
as well as
\[\mathbb{D}^{\kappa-1}(f^*) = 0 \quad \textup{and} \quad \xi_{2-\kappa}(\mathcal{E}_f) = 0.\]

Each of these four operators may be extended to local weak Maass forms in a natural way. 

\subsection{Binary integral quadratic forms and cycle integrals}\label{subsec:BinaryAndCycleIntegrals} A binary integral quadratic form $Q(X,Y) = aX^2 + bXY + cY^2$ will be denoted by $Q = [a,b,c]$ throughout. There exists an intensively studied action of $\Gamma$ on $\mathcal{Q}_D$ which is given by
\[\left(Q \circ \begin{pmatrix}
    a_1 & a_2 \\ a_3 & a_4
\end{pmatrix}\right)(X,Y) = Q(a_1X + a_2Y, a_3X + a_4Y)\]
and is compatible with the action of $\Gamma$ on $\H$ via
\[(Q \circ \gamma)(z,1) = j(\gamma,z)^2Q(\gamma z,1).\]

Due to this property, the Maass and differential operators from above interact nicely with the cycle integral as was shown by Alfes and Schwagenscheidt.

\begin{thm}[\protect{\cite{AlfesSchwagen}*{Theorem 1.1}}] \label{thm:AlfesSchwagen}
    Let $F: \H \to \C$ be a smooth function which transforms like a modular form of weight $2 - 2\kappa, \kappa \in \Z,$ for $\Gamma$. Then the identity
    \[\mathcal{C}_{-2\kappa}(L_{2-2\kappa}F,Q) = \mathcal{C}_{4-2\kappa}(R_{2-2\kappa}F,Q) = \overline{\mathcal{C}_{2\kappa}(\xi_{2-2\kappa}F,Q)}\]
    holds.

    Furthermore, if $F$ is a weak Maass form of weight $2 - 2\kappa$ with eigenvalue $\lambda$, we have
    \begin{align}\label{eq:AlfesSchwagenRaising}
        \mathcal{C}_{2-2\ell}(R_{2-2\kappa}^{\kappa-\ell}F,Q) &= \left((\kappa + \ell)(\kappa - \ell - 1) - \lambda\right) \mathcal{C}_{-2-2\ell}(R_{2-2\kappa}^{\kappa-\ell-2}F,Q), &\ell \leq \kappa - 2;\\
        \label{eq:AlfesSchwagenLowering}
        \mathcal{C}_{2\ell-2}(L_{2-2\kappa}^{2-\kappa-\ell}F,Q) &= \left((\kappa + \ell)(\kappa - \ell - 1) - \lambda\right) \mathcal{C}_{2+2\ell}(L_{2-2\kappa}^{-\kappa-\ell}F,Q), &\ell \leq -\kappa.
    \end{align}
    Here, we adjust the weight while iterating either operator.
\end{thm}

\subsection{Bivariate functions transforming like modular forms} Recall the functions $H_{k_1,k_2}$ from the introduction. The main idea in \cite{Stum26} was combining the following lemma with Eqn. \eqref{eq:AlfesSchwagenRaising} in Theorem \ref{thm:AlfesSchwagen}.

\begin{lemma}[\protect{\cite{lsmeromorphic}*{Lemma 3.2}}] \label{lem:PeterssonPoincareRaisingLowering}
    For $k_1 \in \N_{\geq 2}$ and $k_2 \in \Z$, it holds that
    \begin{align*}
        R_{-2k_2,\tau}H_{k_1,k_2}(z,\tau) &= (k_1 - k_2)H_{k_1,k_2-1}(z,\tau),\\
        L_{-2k_2,\tau}H_{k_1,k_2}(z,\tau) &= (k_1 + k_2)H_{k_1,k_2+1}(z,\tau).
    \end{align*}
\end{lemma}

As it was pointed out in \cite{Stum26}, these identities are imitated by $\mathscr{P}_{k_1,k_2}$ in the sense that
\begin{align}
R_{-2k_2,\tau}\mathscr{P}_{k_1,k_2}(z,\tau) &= (k_1 - k_2)\mathscr{P}_{k_1,k_2-1}(z,\tau),\label{eq:RaisingMaassOnSerifenP}\\
L_{-2k_2,\tau}\mathscr{P}_{k_1,k_2}(z,\tau) &= (k_1 + k_2)\mathscr{P}_{k_1,k_2+1}(z,\tau)\notag
\end{align}
for $k_1 \in \N_{\geq 2}$ and $k_2 \in \Z$.

\section{Extending the Diagram}\label{sec:ExtendingTheDiagram}

The first objective is to establish the commutativity of the diagram at the end of the introduction concerning the functions from Definition \ref{defn:CurlyJIota}. As mentioned by Bringmann and Kane in \cite{BringKane}*{Subsection 2.2}, we have the following technical lemma at our disposal.

\begin{lemma}[\cite{BringKane}*{Lemmata 2.1, 2.2, 2.3, 2.4, 2.5}]\label{lem:BringKaneLemmata}
    The following statements are true.
    \begin{enumerate}
        \item Suppose $f$ is a weak Maass form of weight $\kappa \in \Z$ with eigenvalue $\lambda$. Then $L_\kappa f$ is a weak Maass form of weight $\kappa - 2$ with eigenvalue $\lambda - \kappa + 2$ and $R_\kappa f$ is a weak Maass form of weight $\kappa + 2$ with eigenvalue $\lambda + \kappa$.
        \item If $f$ transforms like a modular form of weight $\kappa$, then the function $y^\kappa \overline{f(z)}$ transforms like a modular form of weight $-\kappa$.
        \item If $f$ is real-analytic in some neighbourhood of $z$, we have
        \[
        y^{2-\kappa}\overline{\xi_\kappa f(z)} = L_\kappa f(z)
        \]
        as well as
        \[
        R_{-\kappa}\left(y^\kappa\overline{f(z)}\right) = \xi_\kappa f(z).
        \]
        \item If $f$ satisfies $\Delta_\kappa = \lambda f$ in some neighbourhood of $z$, then
        \[\Delta_{-\kappa}\left(y^\kappa\overline{f(z)}\right) = (\overline{\lambda} + \kappa)y^\kappa\overline{f(z)}.\]
    \end{enumerate}
\end{lemma}

These relations imply the following proposition establishing some basic properties of $\mathcal{J}_{-1-k,D}$ and $\iota_{k+2,D}$ resembling the ones proven by Bringmann and Kane for $\mathcal{W}_{-k,D}$ and $\omega_{k+1,D}$.

\begin{prop}\label{prop:BasicPropsCurlyJIota}
    Let $k > 2$ and $D > 0$ a fundamental discriminant. The following statements are true:
    \begin{enumerate}
        \item The function $\iota_{k+2,D}$ is a local weak Maass form of weight $2k + 4$ with eigenvalue $2k + 2$ and exceptional set $E_D$. Furthermore, we have the representation
        \begin{equation*}
            \iota_{k+2,D}(z) = \sum_{Q \in \mathcal{Q}_D} \frac{|Q_z|}{Q(z,1)^{k+2}}
        \end{equation*}
        for $z \in \H \setminus E_D$.
        \item The function $\iota_{k+2,D}$ is a locally almost holomorphic modular form of weight $2k + 4$ and a splitting into a locally holomorphic and a locally non-holomorphic part is given by
        \[\iota_{k+1,D}(z) = \frac{g_{k+1,D}(z)}{y} - i\sum_{Q \in \mathcal{Q}_D}\frac{\sgn(Q_z)Q'(z,1)}{Q(z,1)^{k+2}}.\]
        \item The function $\mathcal{J}_{-1-k,D}$ is a local weak Maass form of weight $-2-2k$ with eigenvalue $2k + 2$ and exceptional set $E_D$.
        \item We have
        \[\xi_{-2-2k}\mathcal{J}_{-1-k,D}(z) = (2k + 2)\iota_{k+2,D}(z).\]
        and
        \[\xi_{2k+4}\iota_{k+2,D}(z) = -\mathcal{J}_{-1-k,D}(z).\]
    \end{enumerate}
\end{prop}

\begin{proofa}
    We first notice that the differential operators only depend on local circumstances and hence we may apply Lemma \ref{lem:BringKaneLemmata} to these local weak Maass forms.
    \begin{enumerate}
        \item The first part follows directly by item (1) of Lemma \ref{lem:BringKaneLemmata} and the exceptional set is given by the same property applying to $g_{k+1,D}$ due to \cite{mo21}*{item (2) of Theorem 1.1}. In each connected component of $\H \setminus E_D$, the expressions $\sgn(Q_z)$ are constant and hence we obtain the representation of $\iota_{k+2,D}$ analogously to Theorem \ref{thm:BringKaneOmegaFRelation} since $|Q_z| = \sgn(Q_z)Q_z$.
        \item This is obtained via directly applying the operator $R_{2k+2}$ to $g_{k+1,D}$.
        \item Applying items (2) and (4) of Lemma \ref{lem:BringKaneLemmata} to the definition of $\mathcal{J}_{-1-k,D}$ yields the claim as $g_{k+1,D}$ is locally holomorphic of weight $2k + 2$ and thus $\Delta_{2k+2}g_{k+1,D}(z) = 0$. The exceptional set is given in the same manner as above.
        \item The operator
        \[f(z) \mapsto y^\kappa\overline{f(z)}\]
        is an involution and hence the first identity follows from item (3) of Lemma \ref{lem:BringKaneLemmata}. Using the identity $\Delta_\kappa = -\xi_{2-\kappa} \circ \xi_\kappa$, we obtain
        \begin{align*}
            (2k + 2)\mathcal{J}_{-1-k,D}(z) &= \Delta_{-2 - 2k}\mathcal{J}_{-1-k,D}(z) = -(\xi_{2k + 4} \circ \xi_{-2-2k})\mathcal{J}_{-1-k,D}(z)\\
            &= -(2k + 2)\xi_{2k+4}\iota_{k + 2,D}(z),
        \end{align*}
        as claimed. \hfill $\square$
    \end{enumerate}
\end{proofa}

Analogously to the construction of $\mathcal{F}_{1-k,D}$ by Bringmann, Kane, and Kohnen in \cite{bkk}, Bringmann and Mono \cite{brimo} constructed the function
\[
\mathcal{G}_{-k,D}: \H \to \C, \quad z \mapsto \frac{1}{2}\sum_{Q \in \mathcal{Q}_D}Q(z,1)^k\beta\left(\frac{Dy^2}{|Q(z,1)|^2};k + \frac{1}{2},\frac{1}{2}\right)
\]
which is a locally harmonic Maass form of weight $-2k$ and a preimage of $g_{k+1,D}$ under both the shadow and the Bol operator. With this in hand, we offer the following proposition.

\begin{prop}\label{prop:LoweringOpCurlyG}
    It holds that
    \[\mathcal{J}_{-1-k,D}(z) = D^{-\frac{1}{2}-k}L_{-2k}\mathcal{G}_{-k,D}(z)\]
    for $z \in \H \setminus E_D$.
\end{prop}

\begin{proof}
    Invoking item (3) of Lemma \ref{lem:BringKaneLemmata} and \cite{brimo}*{item (2) of Theorem 1.3} yields the claim.
\end{proof}

\begin{rmk}
    Item (2) of Theorem 1.3 in \cite{brimo} established a representation of $\mathcal{G}_{-k,D}$ as sum of a holomorphic and a non-holomorphic Eichler integral of $g_{k+1,D}$ in the following sense:
    \[\mathcal{G}_{-k,D}(z) = c_\infty - \frac{D^{k + \frac{1}{2}}(2k)!}{(4\pi)^{2k+1}}\mathcal{E}_{g_{k+1,D}}(z) + D^{k + \frac{1}{2}} g_{k+1,D}^*(z)\]
    where $c_\infty$ is a certain constant. While the holomorphic part vanishes under $L_{-2k}$, the non-holomorphic part transmutes into $\mathcal{J}_{-1-k,D}$ by standard arguments.
\end{rmk}

\section{Representations as Traces of Cycle Integrals}

In this section, we exploit the relations given with respect to the various differential operators to obtain representations of functions as traces of cycle integrals using those functions for which such a representation is already known. One major resource is provided by the following theorem.

\begin{thm}[\cite{Stum26}*{Theorem 1.1}]\label{thm:Stum26Result}
    Let $k \in 2\N_{\geq 2}$ and, for $m \in 2\N$ with $m \leq 2k$, the constant $K_m$ be given by
    \[ \frac{2\Gamma(2k + 2)\prod_{j = 0}^{k-1}\left(k + j + 1\right)}{D^{\frac{k + 1}{2}}\Gamma\left(\frac{k+1}{2}\right)^2\prod_{l = 1 - k}^{-1}\left((1 + l)(-l) + k(k + 1)\right)}\prod_{l = 0}^{\frac{m}{2} - 1}\frac{(l + 1)}{(k - l)}.\]
    We have
    \[g_{k+1,D}(z) = K_m \sum_{Q \in \mathcal{Q}_D/\Gamma} \mathcal{C}_{2(m-k)}\left(H_{k+1,k-m}(z, \tau),Q\right)\]
    for $m \in 2\N$ with $m \leq 2k$ and
    \begin{multline*}
        f_{k,D}(z) = \frac{(-1)^{\frac{m}{2}}(2k - 1)}{D^{\frac{k}{2}}} \binom{2k - 2}{k - 1} \left(\prod_{l = 0}^{\frac{m}{2}-1}\frac{2(k + l) + 1}{2l + 1}\right)\\
        \times \sum_{Q \in \mathcal{Q}_D/\Gamma}\mathcal{C}_{2(k+m)}\left(\mathscr{P}_{k,-(k+m)}(z, \tau), Q\right)
    \end{multline*}
    for each $m \in 2\N$.
\end{thm}

\subsection{Weak Maass forms} In the case of $\mathcal{W}_{-k,D}$ and $\omega_{k+1,D}$, the second identity of Theorem \ref{thm:Stum26Result} poses as a crucial ingredient. Using the relations proven by Bringmann and Kane in \cite{BringKane}, we present the following proof.

\begin{proof}[Proof of Theorem \ref{thm:MainResultWeakMaassForms}]
    We obtain
    \begin{align*}
        \mathcal{W}_{-k,D}(z)&= y^{2k} \overline{f_{k,D}(z)}
        \\&= -\frac{(2k - 1)}{D^{\frac{k}{2}}} \binom{2k - 2}{k - 1} (2k + 1)y^{2k} \sum_{Q \in \mathcal{Q}_D/\Gamma} \overline{\mathcal{C}_{4+2k}\left(\mathscr{P}_{k,-(k+2)}(z, \tau), Q\right)},
    \end{align*}
    when using Theorem  \ref{thm:Stum26Result} for $m=2.$ In order to apply Theorem \ref{thm:AlfesSchwagen}, we rewrite 
 $\mathscr{P}_{k,-(k+2)}(z,\tau)$ as a result of applying the raising operator $R_{2k+2,\tau}$ to a function $F$. By Eqn. \eqref{eq:RaisingMaassOnSerifenP}, we have
 \begin{align*}
     R_{2k+2,\tau} \mathscr{P}_{k,-(k+1)}(z,\tau)= (2k+1)\mathscr{P}_{k,-(k+2)}(z,\tau). 
 \end{align*}
 From this, we obtain
 \begin{align*}
     \mathcal{W}_{-k,D}(z)&= -\frac{k}{2D^{\frac{k}{2}}} \binom{2k}{k} y^{2k}\sum_{Q \in \mathcal{Q}_D/\Gamma} \overline{\mathcal{C}_{4+2k}\left(R_{2k+2,\tau}\mathscr{P}_{k,-(k+1)}(z, \tau), Q\right)}
     \\ &= - \frac{k}{2D^{\frac{k}{2}}} \binom{2k}{k} y^{2k}\sum_{Q \in \mathcal{Q}_D/\Gamma}  \mathcal{C}_{-2k}\left(\xi_{2k+2,\tau} \mathscr{P}_{k,-(k+1)}(z, \tau),Q\right).
 \end{align*}
 A direct computation gives
 \begin{align*}
    \xi_{2k+2,\tau}\mathscr{P}_{k,-(k+1)}(z,\tau)= -v^{2k} \overline{L_{2k+2,\tau} \mathscr{P}_{k,-(k+1)}(z,\tau)} =  v^{2k} \overline{\mathscr{P}_{k,-k}(z,\tau)}
\end{align*}
from which we get the desired result
\begin{equation*}
\mathcal{W}_{-k,D}(z)=-\frac{ky^{2k}}{2D^{\frac{k}{2}}} \binom{2k}{k}
\sum_{Q \in \mathcal{Q}_D/\Gamma} \mathcal{C}_{-2k}\left(v^{2k}\overline{\mathscr{P}_{k,-k}\left( z, \tau \right)},Q\right).
\end{equation*}

For $\omega_{k+1,D}(z)$ we proceed in a similar way. 
    We have
    \begin{align*}
       2k \omega_{k+1,D}(z)&= R_{2k,z}f_{k,D}(z)
       \\&=\frac{(2k - 1)}{D^{\frac{k}{2}}} \binom{2k - 2}{k - 1} \sum_{Q \in \mathcal{Q}_D/\Gamma}\mathcal{C}_{2k}\left(R_{2k,z}\mathscr{P}_{k,-k}(z,\tau), Q\right),
\intertext{now using that $\mathscr{P}_{k,-k}(z, \tau)=\mathscr{P}_{k,-k}(\tau,z)$ due to \cite{Stum26}*{Lemma 3.5} we get}
       &= \frac{(-1)^k k}{D^{\frac{k}{2}}} \binom{2k}{k} \sum_{Q \in \mathcal{Q}_D/\Gamma}\mathcal{C}_{2k}\left(R_{2k,z}H_{k,-k}(-\tau,z), Q\right)
       \\&= \frac{(-1)^k k}{D^{\frac{k}{2}}} \binom{2k}{k} \sum_{Q \in \mathcal{Q}_D/\Gamma}\mathcal{C}_{2k}\left(2kH_{k,-(k+1)}(-\tau,z), Q\right)
    \end{align*}
    by using Lemma \ref{lem:PeterssonPoincareRaisingLowering}. Dividing by $2k$ gives the desired result for $\omega_{k+1,D}(z)$.

    The second representation of $\mathcal{W}_{-k,D}$ is derived from the result by Löbrich and Schwagenscheidt \cite{lsmeromorphic}*{Theorem 4.2}. Using Theorem \ref{thm:RelationWf} and Eqn. \eqref{eq:CycleIntegralsCurlyF},  we get
    \begin{align*}
        \mathcal{W}_{-k,D}(z) &= D^{\frac{1}{2}-k}L_{2-2k,z}\mathcal{F}_{1-k,D}(z)\\
        &= \frac{D^{-\frac{k}{2}}}{2}\binom{2k - 2}{k - 1}L_{2-2k,z}\sum_{Q \in \mathcal{Q}_D \slash \Gamma}\mathcal{C}_{2k}(H_{k,k-1}(\tau,z),Q)\\
        &= \frac{D^{-\frac{k}{2}}}{2}\binom{2k - 2}{k - 1}\sum_{Q \in \mathcal{Q}_D \slash \Gamma}\mathcal{C}_{2k}(L_{2k - 2,z}H_{k,k-1}(\tau,z),Q)\\
        &= \frac{k}{4D^{\frac{k}{2}}}\binom{2k}{k}\sum_{Q \in \mathcal{Q}_D \slash \Gamma}\mathcal{C}_{2k}(H_{k,k}(\tau,z),Q),
    \end{align*}
    yielding the claim.
\end{proof}

\subsection{Local weak Maass forms} Similarly to above, we may capitalise on the first identity given in Theorem \ref{thm:Stum26Result} in order to prove the second main result.

\begin{proof}[Proof of Theorem \ref{thm:MainResultLocallyWeakMaassForms}]
    Note that by assumption $2k > 4$. We obtain
    \begin{align*}
        \mathcal{J}_{-1-k,D}(z) &= y^{2k+2}\overline{g_{k+1,D}(z)}\\
        &= K_2y^{2k+2} \sum_{Q \in \mathcal{Q}_D / \Gamma} \overline{\mathcal{C}_{4-2k}\left(H_{k+1,k-2}(z, \tau),Q\right)}
    \end{align*}
    by invoking Theorem \ref{thm:Stum26Result} for $m = 2$. We apply Theorem \ref{thm:AlfesSchwagen} with $F = H_{k+1,k}$, combined again with Lemma \ref{lem:PeterssonPoincareRaisingLowering} which states that
    \begin{equation}
        R_{2-2k,\tau}H_{k+1,k-1}(z,\tau) = 2H_{k+1,k-2}(z,\tau).
    \end{equation}
    Proceeding as in the proof of Theorem \ref{thm:MainResultWeakMaassForms}, we get
    \begin{align*}
        \mathcal{J}_{-1-k,D}(z) &= K_2y^{2k+2} \sum_{Q \in \mathcal{Q}_D / \Gamma} \overline{\mathcal{C}_{4-2k}\left(H_{k+1,k-2}(z, \tau),Q\right)}\\
        &= \frac{K_2}{2}y^{2k+2} \sum_{Q \in \mathcal{Q}_D / \Gamma} \overline{\mathcal{C}_{4-2k}\left(R_{2-2k,\tau}H_{k+1,k-1}(z, \tau),Q\right)}\\
        &= \frac{K_2}{2}y^{2k+2} \sum_{Q \in \mathcal{Q}_D / \Gamma} \mathcal{C}_{2k}\left(\xi_{2-2k,\tau}H_{k+1,k-1}(z, \tau),Q\right)\\
        \intertext{and, using $\xi_{\kappa,\tau} = -v^{\kappa - 2}\overline{L_{\kappa,\tau}}$, we arrive at (compare to Löbrich--Schwagenscheidt \cite{lslocallyharmonic}*{Section 3}) }
        &= -\frac{K_2}{2}y^{2k+2} \sum_{Q \in \mathcal{Q}_D / \Gamma} \mathcal{C}_{2k}\left(v^{-2k}\overline{L_{2-2k,\tau}H_{k+1,k-1}(z, \tau)},Q\right)\\
        &= -kK_2y^{2k+2} \sum_{Q \in \mathcal{Q}_D / \Gamma} \mathcal{C}_{2k}\left(v^{-2k}\overline{H_{k+1,k}(z, \tau)},Q\right),
        %\\
        %&= -kK_2y^{2k+2} \sum_{Q \in \mathcal{Q}_D / \Gamma} \mathcal{C}_{2k}\left(v^{-2k} \textcolor{blue}{H_{k+1,k}(-\overline{z},-\overline{ \tau})},Q\right)\\
        %&= -kK_2y^{2k+2} \sum_{Q \in \mathcal{Q}_D / \Gamma} \mathcal{C}_{2k}\left(-v^{-2k}H_{k+1,k}(\overline{z},\overline{ \tau}),Q\right)\\
        %&=-kK_2y^{2k+2} \sum_{Q \in \mathcal{Q}_D / \Gamma} \mathcal{C}_{2k}\left(H_{k+1,-k}(\overline{z}, \tau),Q\right)
    \end{align*}
    which is the claimed identity. 
    \end{proof}
\begin{rmk}
   Using the definition of $\iota_{k+2}(z)$ we find
       \begin{align*}
       (2k + 2) \iota_{k+2,D}(z) &=R_{2k+2}g_{k+1,D}(z)
       \\&= K_0 \sum_{Q \in \mathcal{Q}_D \slash \Gamma} \mathcal{C}_{-2k}(R_{2k+2,z}H_{k+1,k}(z,\tau),Q),
    \end{align*}
    and have
  \begin{align*}
    R_{2k+2,z} &\left(\sum_{\gamma\in\Gamma} \frac{v^{2k+1}}{(z-\tau)\left(z-\overline{\tau} \right)^{2k+1}}\bigg|_{2k+2,z} \gamma\right)= \sum_{\gamma\in\Gamma} R_{2k+2}\left(\frac{v^{2k+1}}{(z-\tau)\left(z-\overline{\tau} \right)^{2k+1}}\right)\bigg|_{2k+4,z} \gamma\\
    &= \sum_{\gamma\in\Gamma} \left(\frac{-4iv^{2k+1}(z-u)}{(z-\tau)^2\left(z-\overline{\tau} \right)^{2k+2}}+ \frac{-4ikv^{2k+1}}{(z-\tau)\left(z-\overline{\tau} \right)^{2k+2}}+ \frac{(2k+2)v^{2k+1}}{y(z-\tau)\left(z-\overline{\tau} \right)^{2k+1}}\right) \bigg|_{2k+4,z}\gamma.
\end{align*}     
\end{rmk} 

\subsection{Zagier's $f_{k,D}$ function}

The crucial tools in this last subsection will be Bol's identity
\begin{equation}\label{eq:BolsIdentityRauschStum}
	\mathbb{D}^{\kappa - 1} = (-4\pi)^{1 - \kappa}R_{2-\kappa}^{\kappa - 1}
\end{equation}
for $\kappa = 2k$ combined with

\begin{lemma}[\cite{bkk}*{Theorem 1.2}]\label{lem:bkk}
	Let $k > 2$ be even and $D>0$ be a discriminant. Then
	\begin{align*}
	\mathbb{D}^{2k - 1}\mathcal{F}_{1-k,D}(z) &= -\frac{D^{k-\frac{1}{2}}\Gamma(2k - 1)}{(4\pi)^{2k-1}}f_{k,D}(z) \qquad \text{and}\\
	\xi_{2-2k}\mathcal{F}_{1-k,D}(z) &= D^{k-\frac{1}{2}}f_{k,D}(z)
	\end{align*}
	for $z \not \in E_D$.
\end{lemma}

\begin{proof}[Proof of Theorem \ref{thm:NewRepForFkD}]
	A straight-forward computation using Lemma \ref{lem:bkk}, Eqn. \eqref{eq:BolsIdentityRauschStum}, Eqn. \eqref{eq:CycleIntegralsCurlyF}, and Lemma \ref{lem:PeterssonPoincareRaisingLowering} yields
	\begin{align*}
		f_{k,D}(z) &= -\frac{(4\pi)^{2k-1}}{D^{k-\frac{1}{2}}\Gamma(2k - 1)}\mathbb{D}^{2k-1}\mathcal{F}_{1-k,D}(z)\\
		&= -\frac{(4\pi)^{2k-1}}{D^{k-\frac{1}{2}}\Gamma(2k - 1)}(-4\pi)^{1 - 2k}R_{2-2k,z}^{2k - 1}\mathcal{F}_{1-k,D}(z)\\
		&= \frac{1}{2D^{\frac{k}{2}}\Gamma(2k - 1)}\binom{2k - 2}{k - 1}R_{2-2k,z}^{2k - 1}\sum_{Q \in \mathcal{Q}_D \slash \Gamma}\mathcal{C}_{2k}(H_{k,k-1}(\tau,z),Q)\\
		&= \frac{k}{4D^{\frac{k}{2}}}\binom{2k}{k}\sum_{Q \in \mathcal{Q}_D \slash \Gamma}\mathcal{C}_{2k}(H_{k,-k}(\tau,z),Q)
	\end{align*}
	as claimed. The representation for $\omega_{k+1,D}$ then follows due to Theorem \ref{thm:BringKaneOmegaFRelation} and Lemma \ref{lem:PeterssonPoincareRaisingLowering}.
\end{proof}


\begin{bibsection}
\begin{biblist}

\bib{AlfesSchwagen}{article}{
    author={Alfes, C.},
    author={Schwangenscheidt, M.},
    title= {Identities of cycle integrals of weak Maass forms},
    year={2020},
    journal={Ramanujan J. },
    number={3},
    pages={683--688},
}

\bib{bfkr}{book}{
    author={Bringmann, K.},
    author={Folsom, A.},
    author={Ono, K.},
    author={Rolen, L.},
    title={Harmonic Maass Forms and Mock Modular Forms: Theory and Applications},
    publisher={American Mathematical Society},
    series={American Mathematical Society colloquium publications},
    volume={64},
    year={2017},
    isbn={978-1-470-41944-8},
}

\bib{BringKane}{webpage}{
    author = {Bringmann, K.},
    author = {Kane, B.},
    title = {A direct proof of Mono--Rolen--Stumpenhusen and new constructions via the Maass raising operators},
    year = {2026},
    url = {https://arxiv.org/abs/2606.27212}
}

\bib{bkk}{article}{
   author={Bringmann, K.},
   author={Kane, B.},
   author={Kohnen, W.},
   title={Locally harmonic Maass forms and the kernel of the Shintani lift},
   journal={Int. Math. Res. Not. IMRN},
   date={2015},
   number={11},
   pages={3185--3224},
   issn={1073-7928},
   doi={10.1093/imrn/rnu024},
}

\bib{brimo}{webpage}{
   author={Bringmann, K.},
   author={Mono, A.},
   title={A modular framework of functions of Knopp and indefinite binary quadratic forms},
   url={https://arxiv.org/abs/2208.01451},
   year={2022},
}

\bib{koh85}{article}{
   author={Kohnen, W.},
   title={Fourier coefficients of modular forms of half-integral weight},
   journal={Math. Ann.},
   volume={271},
   date={1985},
   number={2},
   pages={237--268},
   issn={0025-5831},
   doi={10.1007/BF01455989},
}

\bib{koza84}{article}{
   author={Kohnen, W.},
   author={Zagier, D.},
   title={Modular forms with rational periods},
   conference={
      title={Modular forms},
      address={Durham},
      date={1983},
   },
   book={
      series={Ellis Horwood Ser. Math. Appl.: Statist. Oper. Res.},
      publisher={Horwood, Chichester},
   },
   date={1984},
   pages={197--249}
}

\bib{lslocallyharmonic}{article}{
   author={L\"{o}brich, S.},
   author={Schwagenscheidt, M.},
   title={Locally harmonic Maass forms and periods of meromorphic modular
   forms},
   journal={Trans. Amer. Math. Soc.},
   volume={375},
   date={2022},
   number={1},
   pages={501--524},
   issn={0002-9947},
   doi={10.1090/tran/8528},
}

\bib{lsmeromorphic}{article}{
    author = {L\"{o}brich, S.},
    author = {Schwagenscheidt, M.},
    title = {Meromorphic Modular Forms with Rational Cycle Integrals},
    journal = {International Mathematics Research Notices},
    year= {2022},
    pages = {312--342},
    doi = {https://doi.org/10.1093/imrn/rnaa104}
}

\bib{mo21}{article}{
   author={Mono, A.},
   title={Locally harmonic Maass forms of positive even weight},
   journal={Israel J. Math.},
   volume={261},
   date={2024},
   number={2},
   pages={671--694},
}

\bib{MonoThesis}{webpage}{
    author = {Mono, A.},
    title = {Harmonic and locally harmonic Maass forms},
    year = {2023},
    url = {https://kups.ub.uni-koeln.de/75117/},
    note = {PhD thesis}
}

\bib{MonoRolenStum}{webpage}{
    author = {Mono, A.},
    author = {Rolen, L.},
    author = {Stumpenhusen, J.},
    title = {On a Divisor Modular Form and a Theta Lift},
    url = {https://arxiv.org/abs/2509.01378},
    year = {2025}
}

\bib{Stum26}{webpage}{
    author = {Stumpenhusen, J.},
    title = {Modular Forms Related to Real Quadratic Fields as Traces of Cycle Integrals},
    year= {2026}
}

\bib{zagier75}{article}{
   author={Zagier, D.},
   title={Modular forms associated to real quadratic fields},
   journal={Invent. Math.},
   volume={30},
   date={1975},
   number={1},
   pages={1--46},
   issn={0020-9910},
   doi={10.1007/BF01389846},
}

\end{biblist}
\end{bibsection}
\end{document}